\documentclass{article}

\usepackage{amsthm}
\usepackage{amsmath, amssymb}
\usepackage{hyperref}
\usepackage{tikz}
\usetikzlibrary{matrix}
\usepackage[matrix, arrow, curve]{xy}
\usepackage{xcolor}

\newtheorem{Thm}{Theorem}[section]
\newtheorem{Def}[Thm]{Definition}
\newtheorem{Lem}[Thm]{Lemma}

\newtheorem{Kor}[Thm]{Corollary}

\title{A note on Suslin matrices and Clifford algebras}

\author{Tariq Syed  \\
	Institut f{\"u}r Mathematik\\
	Johannes Gutenberg-Universit{\"a}t Mainz\\
	Staudingerweg 9\\
	55128 Mainz, Germany\\
	tariq.syed@gmx.de}

\date{\today
}
 
\begin{document}

\maketitle

\begin{abstract}
We give a conceptual explanation for the somewhat mysterious origin of Suslin matrices. This enables us to generalize the construction of Suslin matrices and to give more conceptual proofs of some well-known results.
\end{abstract}

\tableofcontents

\section{Introduction}

Suslin matrices were orginally introduced by Suslin in \cite{S} in the context of his proof of the famous $n!$-theorem (cf. \cite[Theorem 2]{S}, \cite[Chapter III, \S 7]{L}) and play an important role in the study of unimodular rows and stably free modules. The well-known $n!$-theorem was substantially used in the proofs of breakthrough results on stably free modules (cf. \cite[Theorem 1]{S}, \cite[Theorem 7.5]{FRS}); Suslin matrices were also crucially used in the beautiful resolution of Suslin's cancellation conjecture on projective modules in the smooth case (cf. \cite{F}). Suslin matrices build a deeply profound and truly remarkable bridge between the theory of stably free modules, homotopy theory and the theory of quadratic forms as they can be used to define explicit KO-degree maps in Hermitian K-theory which stabilize to the unit map from the unit sphere spectrum to the Hermitian K-theory spectrum in $\mathbb{A}^1$-homotopy theory (cf. \cite{AF}). Despite the widespread use of Suslin matrices in the proofs of several groundbreaking results and their high importance in several research areas, their origin remains to be somewhat mysterious. It should be expected that any step forward in understanding the origin of these matrices can lead to applications in various mathematical research areas in the future.\\
Now let $R$ be a commutative ring and $P$ be a finitely generated projective $R$-module. It is a classical fact that the Clifford algebra $C(q)$ of the hyperbolic quadratic form $q$ on $H(P \oplus R)$ can be identified with the endomorphism ring $End(\Lambda (P \oplus R))$ of the exterior algebra of $P \oplus R$ (cf. \cite[Chapter IV, \S 2]{K}). As a matter of fact, there exists a canonical isomorphism
\begin{center}
$\Phi: C(q) \xrightarrow{\cong} End(\Lambda (P \oplus R))$
\end{center}
of $\mathbb{Z}/2\mathbb{Z}$-algebras. Of course, if $P=R^n$ for some $n \geq 0$, then $\Phi$ and the choice of an ordered basis $\mathcal{B}$ of the free $R$-module $\Lambda (R^{n+1})$ yield an identification
\begin{center}
$\Phi_{\mathcal{B}}: C(q) \xrightarrow{\cong} M_{2^{n+1}}(R)$,
\end{center}
where $M_{2^{n+1}}(R)$ denotes the set of $2^{n+1} \times 2^{n+1}$-matrices over $R$.\\
For any $a,b \in R^{n+1}$, one may define $2^n \times 2^n$-matrices $S_{n+1}(a,b)$ called Suslin matrices as well as $2^n \times 2^n$-matrices $\overline{S_{n+1}(a,b)}$. It is argued in \cite{C1} that the assignment
\begin{center}
$(a,b) \mapsto \begin{pmatrix}
0 & S_{n+1} (a,b) \\
\overline{S_{n+1}(a,b)} & 0
\end{pmatrix}$
\end{center}
induces an identification $C(q)\cong M_{2^{n+1}}(R)$ of $\mathbb{Z}/2\mathbb{Z}$-algebras; we call this isomorphism $\phi$. However, the proof of this statement is very direct and does not relate $\phi$ to $\Phi$. In this paper we prove the following main result (cf. Theorem \ref{Basis-Phi}):\\\\
\textbf{Theorem.} Let $n \geq 0$. There is an ordered basis $\mathcal{B}$ of $\Lambda (R^{n+1})$ such that $\phi = \Phi_{\mathcal{B}}$.\\\\
The theorem gives a conceptual explanation for the somewhat mysterious origin of Suslin matrices; its proof shows precisely how Suslin matrices arise naturally by means of the classical isomorphism $\Phi$. As a matter of fact, as $\Phi$ is defined for any finitely generated projective $R$-module, we can generalize the construction of Suslin matrices and define generalized Suslin matrices. As indicated above, Suslin matrices have been widely used in the study of stably free modules and we foresee that the generalized Suslin matrices will have applications to more general cancellation questions for finitely generated projective modules; this could include an $n!$-theorem for unimodular elements of finitely generated projective modules. As a first application, we can also give a significantly easier and more conceptual proof of the \textbf{Generalized Key Lemma}, which was recently proven in \cite{JR} with a lot of technical effort. Our methods even extend the results proven in \cite{JR}.\\
The paper is structured as follows: We recall the definition of Suslin matrices in Section \ref{2.1} and basic facts about Clifford algebras in Section \ref{2.2}. Then we define generalized Suslin endomorphisms and prove the main result of this paper in Section \ref{3}.

\section*{Acknowledgements}
The author would like to thank Jean Fasel for helpful comments on this work. The author was funded by the Deutsche Forschungsgemeinschaft (DFG, German Research Foundation) - Project number 461453992.

\section{Preliminaries}

Throughout this paper, $R$ will always denote a commutative ring. For any $n \in \mathbb{N}$, we will freely identify $R^n$ with ${(R^n)}^{\vee}$ and ${(R^{\vee})}^{n}$ in the usual way.

\subsection{Suslin matrices}\label{2.1}

We briefly recall the definition of Suslin matrices, which were introduced in \cite{S}.\\
Let $a=(a_{1},...,a_{n}) \in R^n = {(R^n)}^{\vee}$ and let $b=(b_{1},...,b_{n}) \in R^n$. The Suslin matrix $S_{n}(a,b)$ of $(a,b) \in R^n \oplus R^n$ is defined inductively by setting $S_{n}(a,b) = (a_{1})$ if $n=1$ and

\begin{center}
$S_{n} (a,b) = \begin{pmatrix}
a_{1} {id}_{2^{n-2}} & S_{n-1} (a',b') \\
-{S_{n-1} (b',a')}^{t} & b_{1} {id}_{2^{n-2}}
\end{pmatrix}$
\end{center}

if $n \geq 2$, where $a'=(a_{2},...,a_{n}), b'=(b_{2},...,b_{n}) \in R^{n-1}$. Analogously, one defines the matrix $\overline{S_{n}(a,b)}$ by $\overline{S_{n}(a,b)} = (b_{1})$ if $n=1$ and

\begin{center}
$\overline{S_{n} (a,b)} = \begin{pmatrix}
b_{1} {id}_{2^{n-2}} & -S_{n-1} (a',b') \\
{S_{n-1} (b',a')}^{t} & a_{1} {id}_{2^{n-2}}
\end{pmatrix}$
\end{center}

if $n \geq 2$, where again $a'=(a_{2},...,a_{n}), b'=(b_{2},...,b_{n}) \in R^{n-1}$.

\subsection{Clifford algebras}\label{2.2}

In this section, $P$ will denote a finitely generated projective $R$-module. We let
\begin{center}
$\Lambda (P) := \bigoplus_{i \geq 0} \Lambda^{i}(P)$
\end{center}
be the exterior algebra of $P$. There is an obvious $\mathbb{Z}/2\mathbb{Z}$-grading on $\Lambda (P)$ given by $\Lambda (P) = \Lambda_{0} (P) \oplus \Lambda_{1} (P)$, where $\Lambda_{0} (P) = \bigoplus_{i \geq 0} \Lambda^{2i} (P)$ and $\Lambda_{1} (P) = \bigoplus_{i \geq 0} \Lambda^{2i+1} (P)$. In particular, any endomorphism of $\Lambda (P)$ can be written as a matrix

\begin{center}
$\begin{pmatrix}
A_{0,0} & A_{1,0} \\
A_{0,1} & A_{1,1}
\end{pmatrix}$,
\end{center}

where $A_{i,j}: \Lambda_{i} (P) \rightarrow \Lambda_{j} (P)$ for $i,j \in \{0,1\}$. This gives the ring $End (\Lambda (P))$ of endomorphisms of $\Lambda (P)$ the structure of a $\mathbb{Z}/2\mathbb{Z}$-graded algebra over $R$ by declaring an endomorphism of the form

\begin{center}
$\begin{pmatrix}
A_{0,0} & 0 \\
0 & A_{1,1}
\end{pmatrix}$
\end{center}

of degree $0$ and an endomorphism of the form

\begin{center}
$\begin{pmatrix}
0 & A_{1,0} \\
A_{0,1} & 0
\end{pmatrix}$
\end{center}

of degree $1$.\\
Now consider $H (P \oplus R) := P \oplus R \oplus P^{\vee} \oplus R^\vee$. We have the usual hyperbolic quadratic form on $H (P \oplus R)$ given by $q (p,a,f,b) = f(p) + ab$. Consider the tensor algebra
\begin{center}
$T (H(P \oplus R)) = \bigoplus_{i \geq 0} T^i (H(P \oplus R))$
\end{center}
of $H (P \oplus R)$ over $R$, where $T^0 (H(P \oplus R)) = R$ and, for $i \geq 1$, $T^i (H(P \oplus R))$ is the $i$-fold tensor product of $i$ copies of $H(P \oplus R)$ over $R$. The Clifford algebra is then defined as follows  (cf. \cite[Chapter IV, \S 1]{K}):
\begin{Def}
The Clifford algebra $C(q)$ is the quotient of $T (H(P \oplus R))$ modulo its two-sided ideal generated by the relations $(p, a, f, b) \otimes (p, a, f, b) = q (p,a,f,b)$ for all $(p,a,f,b) \in H (P\oplus R)$.
\end{Def}
It follows immediately from this definition that the Clifford algebra is a $\mathbb{Z}/2\mathbb{Z}$-graded $R$-algebra. It is well-known that there is a canonical isomorphism

\begin{center}
$\Phi: C (q) \rightarrow End (\Lambda (P \oplus R))$
\end{center}

of $\mathbb{Z}/2\mathbb{Z}$-graded $R$-algebras (cf. \cite[Chapter IV, \S 2, Proposition 2.1.1]{K}). This isomorphism is given as follows:\\
Let $x = (p,a,f,b) \in H (P \oplus R)$. Then left multiplication with $x' = (p,a) \in P \oplus R$ determines a homomorphism

\begin{center}
$l_{x'}: \Lambda (P \oplus R) \rightarrow \Lambda (P \oplus R)$.
\end{center}

If we let ${x''} = (f,b) \in P^{\vee} \oplus R^\vee = {(P \oplus R)}^{\vee}$, then ${x''}$ determines a homomorphism

\begin{center}
$d_{x''}: \Lambda (P \oplus R) \rightarrow \Lambda (P \oplus R)$.
\end{center}

given by $d_{x''} (a_{1} \wedge ... \wedge a_{r}) = \sum_{i=1}^{r} {(-1)}^{i+1} {x''}(a_{i})~a_{1} \wedge ... \wedge \widehat{a_{i}}\wedge ... \wedge a_{r}$. Then the endomorphism $l_{x'} + d_{x''}$ satisfies

\begin{center}
${(l_{x'} + d_{x''})}^{2} = q(p,a,f,b) \cdot Id_{\Lambda (P \oplus R)}$
\end{center}

and hence induces the map

\begin{center}
$\Phi: C (q) \rightarrow End (\Lambda (P \oplus R))$,
\end{center}

which is the isomorphism mentioned above. For any $x = (p,a,f,b)$ as above, note that we can write $\Phi (x)$ in the form

\begin{center}
$\begin{pmatrix}
0 & \Phi_{1,0}(x) \\
\Phi_{0,1}(x) & 0
\end{pmatrix}$.
\end{center}

Since ${(\Phi (x))}^{2} = q(p,a,f,b) \cdot Id_{\Lambda (P \oplus R)}$ as mentioned above, it follows easily that 

\begin{center}
$\Phi_{1,0}(x)\Phi_{0,1}(x) = q(p,a,f,b) \cdot Id_{\Lambda_{0}(P \oplus R)}$
\end{center}

and

\begin{center}
$\Phi_{0,1}(x)\Phi_{1,0}(x)=q(p,a,f,b) \cdot Id_{\Lambda_{1}(P \oplus R)}$.
\end{center}

In particular, if $q (p,a,f,b) \in R^{\times}$, then $\Phi_{0,1}(x)$ and $\Phi_{1,0}(x)$ are isomorphisms. If we let $x_{0} = (0,1,0,1) \in H (P \oplus R)$, then $\Phi_{0,1}(x_{0})$ is an isomorphism with inverse $\Phi_{1,0}(x_{0})$.

\section{Generalized Suslin endomorphisms}\label{3}

\begin{Def}
For any $x = (p,a,q,b) \in H (P \oplus R)$, we define

\begin{center}
$S (x): \Lambda_{1} (P \oplus R) \xrightarrow{\Phi_{1,0}(x)} \Lambda_{0} (P \oplus R) \xrightarrow{\Phi_{0,1}(x_{0})} \Lambda_{1} (P \oplus R)$.\\
$\overline{S(x)}: \Lambda_{1} (P \oplus R) \xrightarrow{\Phi_{1,0}(x_{0})} \Lambda_{0} (P \oplus R) \xrightarrow{\Phi_{0,1}(x)} \Lambda_{1} (P \oplus R)$.
\end{center}
We call $S(x)$ the generalized Suslin matrix or the generalized Suslin endomorphism of $x \in H (P \oplus R)$.
\end{Def}

Since $\Lambda^{0}(R)=R$, $\Lambda^{1}(R)=R$ and $\Lambda^{i}(R)=0$ for $i \geq 2$, there are canonical isomorphisms
\begin{center}
$\Lambda^{i} (P \oplus R) \cong (\Lambda^{i} (P)\otimes_{R}\Lambda^{0}(R)) \oplus (\Lambda^{i-1} (P)\otimes_{R}\Lambda^{1}(R)) \cong \Lambda^{i} (P) \oplus \Lambda^{i-1} (P)$
\end{center}
for $i \geq 0$, where we let $\Lambda^{-1}(P)=0$ by convention. As a matter of fact, we therefore obtain canonical isomorphisms
\begin{center}
$\Lambda_{0} (P \oplus R) \cong \Lambda_{0} (P) \oplus \Lambda_{1} (P)$
\end{center}
and
\begin{center}
$\Lambda_{1} (P \oplus R) \cong \Lambda_{0} (P) \oplus \Lambda_{1} (P)$.
\end{center}


An easy direct computation by means of the canonical isomorphisms above then shows that $\Phi_{1,0}(x)$ can be written as the matrix

\begin{center}
$\begin{pmatrix}
b \cdot Id_{\Lambda_{0} (P)} & l_p + d_f \\
l_p + d_f & -a \cdot Id_{\Lambda_{1} (P)}
\end{pmatrix}$
\end{center}

and $\Phi_{0,1} (x)$ can be written as the matrix

\begin{center}
$\begin{pmatrix}
a \cdot Id_{\Lambda_{0} (P)} & l_p + d_f \\
l_p + d_f & -b \cdot Id_{\Lambda_{1} (P)}
\end{pmatrix}$
\end{center}

with respect to the canonical decompositions $\Lambda_{0} (P \oplus R) \cong \Lambda_{0} (P) \oplus \Lambda_{1} (P)$ and $\Lambda_{1} (P \oplus R) \cong \Lambda_{0} (P) \oplus \Lambda_{1} (P)$ mentioned above. In particular, $\Phi_{0,1} (x_{0})$ and $\Phi_{1,0}(x_{0})$ are given by the matrix

\begin{center}
$\begin{pmatrix}
Id_{\Lambda_{0} (P)} & 0 \\
0 & -Id_{\Lambda_{1} (P)}
\end{pmatrix}$.
\end{center}

Altogether, it follows that the endomorphism $S(x)$ is given by the matrix

\begin{center}
$\begin{pmatrix}
b \cdot Id_{\Lambda_{0} (P)} & l_p + d_f \\
-(l_p + d_f) & a \cdot Id_{\Lambda_{1} (P)}
\end{pmatrix}$
\end{center}

and that the endomorphism $\overline{S(x)}$ is given by the matrix

\begin{center}
$\begin{pmatrix}
a \cdot Id_{\Lambda_{0} (P)} & -(l_p + d_f) \\
l_p + d_f & b \cdot Id_{\Lambda_{1} (P)}
\end{pmatrix}$.
\end{center}

Now assume that $P = R^n$ is a free $R$-module of rank $n$. We want to show that our definition of the generalized Suslin matrices actually generalizes Suslin's original definition of Suslin matrices. We need the following small lemma:

\begin{Lem}\label{Bases-Lemma}
Let $n \geq 0$. Assume that there is an ordered basis $\mathcal{B}_1$ of $\Lambda_{1} (R^{n+1})$ such that for all $x = (p,a,f,b) \in R^{n+1} \oplus {(R^{n+1})}^{\vee} = R^{n+1} \oplus R^{n+1}$ the matrices representing $S(x)$ and $\overline{S (x)}$ with respect to $\mathcal{B}_1$ are just $S_{n+1} (b,f,a,p)$ and $\overline{S_{n+1} (b,f,a,p)}$. Then there is an ordered basis $\mathcal{B}_0$ of $\Lambda_{0} (R^{n+1})$ such that for all $x = (p,a,f,b) \in R^{n+1} \oplus {(R^{n+1})}^{\vee} = R^{n+1} \oplus R^{n+1}$ the matrices representing $\Phi_{1,0}(x)$ and $\Phi_{0,1} (x)$ with respect to $\mathcal{B}_1$ and $\mathcal{B}_0$ are just $S_{n+1} (b,f,a,p)$ and $\overline{S_{n+1} (b,f,a,p)}$ as well.
\end{Lem}

\begin{proof}
If $\mathcal{B}_1$ is given by $b_{1},...,b_{2^{n}}$, then just define $\mathcal{B}_0$ to be the ordered basis given by $\Phi_{1,0}(x_{0}) (b_{1}), ..., \Phi_{1,0}(x_{0})(b_{2^{n}})$.
\end{proof}

Now let $n=0$, i.e., $P=0$. Then we have $\Lambda (P \oplus R) = \Lambda (R) = \Lambda^0 (R) \oplus \Lambda^1 (R) = R \oplus R$. Using this decomposition of $\Lambda (R)$, the endomorphism $S(x)$ is just multiplication by $b \in R$ and $\overline{S(x)}$ is just multiplication by $a$. So both endomorphisms correspond to $S_1 (b,a)$ and $\overline{S_1(b,a)}$ respectively.\\
Now assume $n = 1$. Then $\Lambda_{0} (P) = \Lambda_{1} (P) = R$, $p,a \in R$, $f,b \in R^\vee = R$ and it follows that $S(x)$ is in fact given by the matrix

\begin{center}
$\begin{pmatrix}
b & f \\
-p & a
\end{pmatrix}$.
\end{center}

This is just the Suslin matrix $S_{2} (b,f,a,p)$ of the rows $(b,f)$ and $(a,p)$. It also follows that the endomorphism $\overline{S (x)}$ is given by the matrix

\begin{center}
$\begin{pmatrix}
a & -f \\
p & b
\end{pmatrix}$.
\end{center}

This is just the matrix $\overline{S_{2} (b,f,a,p)}$ of the elements $(b,f)$ and $(a,p)$.\\
Now assume $n \geq 2$ and $P = R^{n}$. By induction, we may assume that there is an ordered basis $\mathcal{B}_1$ of $\Lambda_{1} (P)$ such that both $S (p,f): \Lambda_{1} (P) \rightarrow \Lambda_{1} (P)$ and $\overline{S (p,f)}: \Lambda_{1} (P) \rightarrow \Lambda_{1} (P)$ are represented by $S_{n} (f,p)$ and $\overline{S_{n} (f,p)}$ with respect to $\mathcal{B}_1$ for all $(p,f) \in R^{n} \oplus {(R^{n})}^{\vee} = R^{n} \oplus R^{n}$. Consequently, by the lemma above, we conclude that there is an ordered basis $\mathcal{B}_0$ of $\Lambda_{0} (P)$ such that $\Phi_{1,0} (p,f) = l_p + d_f : \Lambda_{1} (P) \rightarrow \Lambda_{0} (P)$ and $\Phi_{0,1} (p,f) = l_p + d_f : \Lambda_{0} (P) \rightarrow \Lambda_{1} (P)$ are represented by $S_{n} (f,p)$ and $\overline{S_{n} (f,p)}$ with respect to $\mathcal{B}_1$ and $\mathcal{B}_0$ for all $x = (p,f) \in R^{n} \oplus {(R^{n})}^{\vee} = R^{n} \oplus R^{n}$. If we then let $\mathcal{B}$ be the ordered basis of $\Lambda_{1} (P \oplus R) \cong \Lambda_{0} (P) \oplus \Lambda_{1} (P)$ obtained by combining $\mathcal{B}_1$ and $\mathcal{B}_0$, then $S (x)$ and $\overline{S (x)}$ are automatically represented by $S_{n+1} (b,f,a,p)$ and $\overline{S_{n+1} (b,f,a,p)}$ with respect to $\mathcal{B}$ for all $x= (p,a,f,b) \in R^{n+1} \oplus {(R^{n+1})}^{\vee} = R^{n+1} \oplus R^{n+1}$. Thus, by induction, we have proven:

\begin{Thm}\label{Justification}
Let $n \geq 0$. There is an ordered basis $\mathcal{B}_1$ of $\Lambda_{1} (R^{n+1})$ such that for all $x = (p,a,f,b) \in R^{n+1} \oplus {(R^{n+1})}^{\vee} = R^{n+1} \oplus R^{n+1}$ the matrices representing $S(x)$ and $\overline{S (x)}$ with respect to $\mathcal{B}_1$ are just $S_{n+1} (b,f,a,p)$ and $\overline{S_{n+1} (b,f,a,p)}$.
\end{Thm}

This justifies our terminology. We now list some useful properties of the generalized Suslin matrices, which are verified immediately:

\begin{Lem}
Let $P$ be a finitely generated projective $R$-module of rank $n \geq 0$. Then
\begin{itemize}
\item[a)] $S(x)\overline{S (x)} = \overline{S (x)}S(x)= q(x) \cdot id_{\Lambda_{1} (P \oplus R)}$ for $x = (p,a,f,b) \in H (P \oplus R)$.
\item[b)] $S(p,a,f,b)=S(0,0,f,b) + S(p,a,0,0)$ for $x=(p,a,f,b) \in H (P \oplus R)$.
\item[c)] $S (rp,ra,rf,rb)=rS(p,a,f,b)$ for $x=(p,a,f,b) \in H (P \oplus R)$ and $r \in R$.
\item[d)] $S(p_{1},a_{1},0,0) + S(p_{2},a_{2},0,0) = S(p_{1}+p_{2},a_{1}+a_{2},0,0)$ for $p_{1},p_{2} \in P$ and $a_{1},a_{2} \in R$;
\item[e)] $S(0,0,f_{1},b_{1})+S(0,0,f_{2},b_{2})=S(0,0,f_{1}+f_{2},b_{1}+b_{2})$ for $f_{1},f_{2} \in P^\vee$ and $b_{1},b_{2} \in R^\vee = R$;
\item[f)] $det (S(x)) = {(q(x))}^{2^{n-1}}$ for $x = (p,a,f,b) \in H (P \oplus R)$.
\item[g)] If $x = (p,a,f,b) \in H (P \oplus R)$, then $S(x) \in Aut (\Lambda_{1} (P \oplus R))$ if and only if $q(x) \in R^{\times}$.
\item[h)] $\overline{S (x)} = S (-p,b,-f,a)$ for $x=(p,a,f,b) \in H (P \oplus R)$.
\end{itemize}
\end{Lem}

\begin{Thm}\label{Basis-Phi}
Let $n \geq 0$. There are ordered bases $\mathcal{B}_1$ of $\Lambda_{1} (R^{n+1})$ and $\mathcal{B}_0$ of $\Lambda_{0} (R^{n+1})$ such that for all $x = (p,a,f,b) \in R^{n+1} \oplus {(R^{n+1})}^{\vee} = R^{n+1} \oplus R^{n+1}$ the matrices representing $\Phi_{1,0}(x)$ and $\Phi_{0,1} (x)$ with respect to $\mathcal{B}_1$ and $\mathcal{B}_0$ are just $S_{n+1} (b,f,a,p)$ and $\overline{S_{n+1} (b,f,a,p)}$.
\end{Thm}

\begin{proof}
Follows from Lemma \ref{Bases-Lemma} and Theorem \ref{Justification}.
\end{proof}

By means of the theory of generalized Suslin matrices, we are able to give a conceptual explanation for the \textbf{Generalized Key Lemma} proven in \cite{JR}:

\begin{Thm}\label{GeneralizedKeyLemma-1}
Let $n \geq 0$. For all $x = (p,a,f,b) \in R^{n+1} \oplus {(R^{n+1})}^{\vee} = R^{n+1} \oplus R^{n+1}$ and $\varphi \in SL_{n+1}(R)$, one has
\begin{center}
$\Phi_{1,0}(\varphi^{-1}(p,a),(f,b)\circ \varphi) = \Lambda_{0}({\varphi}^{-1})\Phi_{1,0}(x)\Lambda_{1}(\varphi)$.
\end{center}
\end{Thm}
\begin{proof}
A direct computation shows that indeed
\begin{center}
$l_{\varphi^{-1}(p,a)} + d_{(f,b)\circ \varphi} = \Lambda_{0}({\varphi}^{-1})(l_{(p,a)} + d_{(f,b)}) \Lambda_{1}(\varphi)$.
\end{center}
\end{proof}

Theorem \ref{GeneralizedKeyLemma-1} recovers the statement proven in \cite{JR}:

\begin{Kor}\label{GeneralizedKeyLemma-2}
Let $n \geq 0$. For all $x = (p,a,f,b) \in R^{n+1} \oplus {(R^{n+1})}^{\vee} = R^{n+1} \oplus R^{n+1}$ and $\varphi \in SL_{n+1}(R)$, one has
\begin{center}
$S_{n+1}((b,f) \circ \varphi, \varphi^{-1}(a,p)) = B_{\varphi}^{-1} S_{n+1}(b,f,a,p)A_{\varphi}$,
\end{center}
where $A_{\varphi}$ and $B_{\varphi}$ are the matrices representing the homomorphisms $\Lambda_{1}(\varphi)$ and $\Lambda_{0}(\varphi)$ respectively for some chosen bases of $\Lambda_{1}(R^{n+1})$ and of $\Lambda_{0}(R^{n+1})$.
\end{Kor}

\begin{proof}
This follows from Theorem \ref{GeneralizedKeyLemma-1} and Theorem \ref{Basis-Phi}.
\end{proof}

Note that we do not assume that $q(x)=1$ as in \cite[Theorem 3.3]{JR}.

\end{document}